\documentclass[11pt]{amsart}
\usepackage{amsmath,amssymb,amsthm,amscd}
\usepackage{colordvi,graphicx}
\usepackage[margin=1in]{geometry}
\usepackage{hyperref}
\usepackage[enableskew]{youngtab}
\usepackage[all,cmtip]{xy}
\usepackage{color}
\usepackage{qtree}
\usepackage{ulem} 

\newtheorem{theorem}{Theorem}[section]
\newtheorem{lemma}[theorem]{Lemma}
\newtheorem{proposition}[theorem]{Proposition}
\newtheorem{corollary}[theorem]{Corollary}

\theoremstyle{definition}

\newtheorem{example}[theorem]{Example}

\theoremstyle{remark}
\newtheorem{remark}[theorem]{Remark}

\numberwithin{table}{section}
\numberwithin{figure}{section}

\def\ZZ{{\mathbb Z}}\def\CC{{\mathbb C}}\def\FF{{\mathbb F}}
\def\comb{\mathrm{comb}}\def\qand{\quad\text{and}\quad}

\newcommand{\nola}{{\rule{2pt}{0pt}}} 
\renewcommand{\SS}{{\mathfrak S}}

\def\P{\mathcal P} \def\tC{\widetilde{C}} 
\newcommand{\T}{{\mathcal T}} 
 
\DeclareMathOperator{\ld}{\delta} \DeclareMathOperator{\rd}{\rho}

\newcommand{\defcolor}[1]{{\color{blue}{#1}}}
\newcommand{\demph}[1]{\defcolor{\it #1}}

\title{Variations of the Catalan numbers from some nonassociative binary operations}

\author{Nickolas Hein}
\address{Department of Mathematics and Computer Science, Benedictine College, Atchison, KS 66002, USA}
\email{nhein@benedictine.edu}

\author{Jia Huang}
\address{Department of Mathematics and Statistics, University of Nebraska at Kearney, Kearney, NE 68849, USA}
\email{huangj2@unk.edu}

\thanks{The authors thank Brendon Rhoades for giving a proof of Proposition~\ref{prop:NumberOfIdeals} and allowing us to include it in this paper, and thank Alex Miller for helpful conversations.}

\keywords{Binary operation, binary tree, Catalan number, nonassociativity.}

\begin{document}

\begin{abstract}
We investigate certain nonassociative binary operations that satisfy a four-parameter generalization of the associative law. 
From this we obtain variations of the ubiquitous Catalan numbers and connections to many interesting combinatorial objects such as binary trees, plane trees, lattice paths, and permutations.
\end{abstract}

\maketitle

\section{Introduction}

\subsection{Background and Motivation}
Binary operations are widely used in mathematics and other fields.
Some operations are associative, including addition, multiplication, union, intersection, and function composition.
Others are not, such as subtraction, division, exponentiation, vector cross product, and Lie algebra multiplication.
We consider a natural question: to what degree is a given operation nonassociative?

We use $*$ to denote a binary operation on a set $A$ and $a_i$, for $i\in \mathbb N$, to denote an $A$-valued indeterminate.
Let $\defcolor{\P_{*,n}}$ be the set of all parenthesizations of the otherwise ambiguous expression $a_0*\cdots*a_n$.
The set $\P_{*,n}$ is in bijection with the set of (full) binary trees with $n+1$ leaves, denoted by $\defcolor{\T_n}$.
We illustrate $\P_{*,3}\leftrightarrow \T_3$ below.
\[\scriptsize\begin{array}{ccccc}
\Tree [.  [. [. 0 1 ] 2 ] 3 ] &
\Tree [.  [. 0 1 ] [. 2 3 ] ] &
\Tree [.  [. 0 [. 1 2 ] ] 3 ] &
\Tree [. 0 [. [. 1 2 ] 3 ] ] &
\Tree [. 0  [. 1 [. 2 3 ] ] ] \\
\updownarrow & \updownarrow & \updownarrow & \updownarrow & \updownarrow \\
\rule{5pt}{0pt}((a_0 {*} a_1) {*} a_2) {*} a_3\rule{5pt}{0pt} &
\rule{5pt}{0pt}(a_0 {*} a_1) {*} (a_2 {*} a_3)\rule{5pt}{0pt} &
\rule{5pt}{0pt}(a_0 {*} (a_1{*} a_2)) {*} a_3 \rule{5pt}{0pt} &
\rule{5pt}{0pt}a_0 {*} ((a_1 {*} a_2) {*} a_3)\rule{5pt}{0pt} &
\rule{5pt}{0pt}a_0 {*} (a_1 {*} (a_2 {*} a_3))\rule{5pt}{0pt}
\end{array}\]
One reason $\P_{*,n}$ and $\T_n$ must be in bijection is that they are each enumerated by the \demph{Catalan number} $\defcolor{C_n}:=\frac{1}{n+1} {2n\choose n}$ which does not depend on $*$.

In this paper we investigate two nonassociativity measurements, viewing the failure of a binary operation to be associative through the lens of its inequivalent parenthesizations.
As such, we broaden our view to other Catalan objects.
We say trees $t,t'\in\T_n$ are \demph{$(*,n)$-equivalent}, written $\defcolor{t \sim_* t'}$, if the corresponding parenthesizations are equal as functions from $A^{n+1}$ to $A$.
This is an equivalence relation on a set of Catalan objects, and for brevity we say its equivalence classes are \demph{$(*,n)$-classes}.
We define $\defcolor{C_{*,n}}$ to be the number of $(*,n)$-classes and immediately observe $1\le C_{*,n}\le C_n$.
We now have an alternate definition of associativity which agrees with the traditional meaning: $*$ is associative if $C_{*,n}=1$ for all $n\in \mathbb N$.
Thus $C_{*,n}$ measures the failure of $*$ to be associative.
We say $*$ is \demph{totally nonassociative} if our measure for nonassociativity attains its theoretical upper bound, $C_{*,n}=C_n$.

Alternatively, one may quantify nonassociativity by computing the cardinality of the largest $(*,n)$-class, denoted $\defcolor{\tC_{*,n}}$, for each $n$.
As with our other measure of nonassociativity, we have $1\le \tC_{*,n}\le C_n$. 
One may see that $\tC_{*,n}=1$ if and only if $*$ is totally nonassociative and $\tC_{*,n}=C_n$ if and only if $*$ is associative.
Moreover, $2\le C_{*,n}+\tC_{*,n}\le C_n+1$.

Cs\'{a}k\'{a}ny and Waldhauser~\cite{AssociativeSpectra1} refered to the sequence $(C_{*,n})_{n=1}^\infty$ as the \demph{associative spectrum}, and Braitt and Silberger~\cite{Subassociative} called it the \demph{subassociativity type}.
Its generalization to $m$-ary operations for $m\ge2$ was introduced by Liebscher and Waldhauser~\cite{m-ary}.
Independently, we proposed the study of both $C_{*,n}$ and $\tC_{*,n}$ in our earlier work~\cite{CatMod}, and Msapatoa generalized our work to $m$-ary operations~\cite{FussCatMod}.

The nonassociativity of a binary operation $*$ has been studied in other ways. For example, Lord~\cite{Lord} introduced the \demph{depth of nonassociativity}, which is given by $\inf \{ n+1: C_{*,n}<C_n\} = \inf \{ n+1: \tC_{*,n} > 1 \}$ and is thus refined by either measurement $C_{*,n}$ or $\tC_{*,n}$.
Another example is the \demph{index of nonassociativity}, which is the cardinality of the set $\{(a,b,c)\in A^3: (a*b)*c\ne a*(b*c)\}$.  This measure, studied by various people, does not appear directly related to $C_{*,n}$ or $\tC_{*,n}$; see~\cite[Section~3]{AlmostAssociative} and the references therein.


Addition and subtraction are among the simplest examples of associative and nonassociative binary operations, respectively.
Cs\'{a}k\'{a}ny--Waldhauser~\cite{AssociativeSpectra1} and Braitt--Silberger~\cite{Subassociative} each showed $C_{-,n}=2^{n-1}$ for all $n\ge1$.
In previous work~\cite{CatMod}, we investigated the nonassociativity of a $1$-parameter family of binary operations which generalize addition and subtraction.
Here, we further generalize to a $4$-parameter family (depending on $d,e,k,\ell$) giving a richer class of examples of binary operations that are neither associative nor totally nonassociative.
We are mainly interested in the sequences of numbers $C_{*,n}$ and $\tC_{*,n}$ for $*$ in this $4$-parameter family.
Our prototypical example is the binary operation $*$ on $\CC[x,y]/I$ given by
\begin{equation}\label{eq:*}
f*g := xf + yg, \quad\forall f,g\in \CC[x,y]/I
\end{equation}
where $I=(x^{d+k}-x^d,y^{e+\ell}-y^e)$ is an ideal of the polynomial ring $\CC[x,y]$.
Though we are presently interested in the ideal given above, one may more generally study binary operations defined by $\eqref{eq:*}$ with $I$ being any ideal.

\subsection{Terminology and notation}
A parenthesization corresponding to $t\in\T_n$ has the form
\begin{equation}\label{eq:expansion}
x^{\ld_0(t)} y^{\rd_0(t)} f_0 + \cdots + x^{\ld_n(t)} y^{\rd_n(t)} f_n. 
\end{equation}
Here we list the leaves of $t$ as $0,1,\ldots,n$ according to \demph{preorder} and define the \demph{left depth} $\defcolor{\ld_i(t)}$ (resp., \demph{right depth} $\defcolor{\rd_i(t)}$) of $i$ to be the number of left (resp., right) steps along the unique path from the root of $t$ down to $i$.
Equivalently, this left (resp., right) depth enumerates the left (resp., right) children in the unique path from the root of $t$ down to $i$ (ignoring the root of $t$ if it is a left child).
The map sending each $t\in\T_n$ to its \demph{left depth} $\defcolor{\ld(t)} := (\ld_0(t),\ldots,\ld_n(t))$ is one-to-one~\cite[\S2.1]{CatMod}.
Symmetrically, the map sending each binary tree $t\in\T_n$ to its \demph{right depth} $\defcolor{\rd(t)} := (\rd_0(t),\ldots,\rd_n(t))$ is also one-to-one.

To characterize $(*,n)$-equivalence for $*$ defined by \eqref{eq:*}, we define some equivalence relations between two sequences $\mathbf{b}=(b_0,\ldots,b_n)$ and $\mathbf{c}=(c_0,\ldots,c_n)$ of nonnegative integers:
\begin{itemize}\itemsep=3pt
\item
$\defcolor{\mathbf{b} \sim_k \mathbf{c}}$ if $b_i\equiv c_i\pmod k$ for $i=0,\ldots,n$, 
\item
$\defcolor{\mathbf{b} \sim^d\mathbf{c}}$ if $\min\{b_i,c_i\}<d$ implies $b_i=c_i$ for $i=0,\ldots,n$, and
\item
$\defcolor{\mathbf{b} \sim_k^d \mathbf{c}}$ if $\mathbf{b} \sim_k \mathbf{c}$ and $\mathbf{b} \sim^d\mathbf{c}$.
\end{itemize}
If $*$ is defined by \eqref{eq:*} then comparing expressions for $t,t'\in\T_n$ of the form \eqref{eq:expansion} implies
\begin{equation}\label{eq:mod-nil-lr}
t\sim_* t' \quad\text{if and only if}\quad \ld(t) \sim_k^d \ld(t') \qand \rd(t) \sim_\ell^e \rd(t')\,.
\end{equation}


More generally, when every equivalent pair of binary trees $t\sim_*t'$ satisfies both $\ld(t) \sim_k^d \ld(t')$ and $\rd(t) \sim_\ell^e \rd(t')$, we say $*$ is \demph{$(k,\ell)$-associative at depth $(d,e)$}.
Note that $(1,1)$-associativity at depth $(1,1)$ is the usual associativity.
We write $\defcolor{C^{d,e}_{k,\ell,n}}:=C_{*,n}$ and $\defcolor{\tC^{d,e}_{k,\ell,n}}:=\tC_{*,n}$ for any binary operation $*$ satisfying \eqref{eq:mod-nil-lr}.

We observe that 
$\mathbf{b}\sim_{k}^d\mathbf{c}$ implies $\mathbf{b}\sim_{k'}^{d'}\mathbf{c}$ if $d\le d'$ and $k\mid k'$.
Thus if $d \le d'$, $e\le e'$, $k \mid k'$, and $\ell \mid \ell'$, then $C^{d,e}_{k,\ell,n}\le C^{d',e'}_{k',\ell',n}$ and $\tC^{d,e}_{k,\ell,n}\ge \tC^{d',e'}_{k',\ell',n}$, and $(d,e)$-associativity at depth $(k,\ell)$ implies $(d',e')$-associativity at depth $(k',\ell')$.

Also, the $(*,n)$-equivalence classes that determine $C_{k,\ell,n}^{d,e}$ and $\tC_{k,\ell,n}^{d,e}$ are the same as the classes that determine $C_{\ell,k,n}^{e,d}$ and $\tC_{\ell,k,n}^{e,d}$, but with each binary tree in each class reflected about a vertical line.
Thus for $d,e,k,\ell\ge1$ we have $C^{d,e}_{k,\ell,n} = C^{e,d}_{\ell,k,n}$ and $\tC^{d,e}_{k,\ell,n} = \tC^{e,d}_{\ell,k,n}$.


We now describe the relationship between $\sim_k^d$ in \eqref{eq:mod-nil-lr} and associativity.
First, note the relation $\sim_1^d$ coincides with $\sim^d$ as all integers are congruent modulo $1$.
Next, observe, since $\ld_i(t)=0 \Leftrightarrow i=n$ and $\rd_i(t)=0 \Leftrightarrow i=0$ for all $t\in\T_n$, we see $\sim_k^1$ coincides with $\sim_k$ on left and right depths of binary trees in $\T_n$.
In earlier work~\cite{CatMod}, we used plane trees, Dyck paths, and Lagrange inversion to determine $\defcolor{C_{k,n}}:=C^{1,1}_{k,1,n}$, which was called a \demph{($k$-)modular Catalan number} as for any binary operation $*$ satisfying \eqref{eq:mod-nil-lr} with $d=e=\ell=1$, the $(*,n)$-relation is the same as the congruence relation modulo $k$ on left depths of binary trees in $\T_n$.
We also determined $\defcolor{\tC_{k,n}} := \tC^{1,1}_{k,1,n}$ and enumerated $(*,n)$-classes with this largest size.
By our earlier result~\cite[Proposition~2.11]{CatMod}, the ``if'' part of \eqref{eq:mod-nil-lr} with $d=e=\ell=1$  is equivalent to \demph{$k$-associativity}, given by the rule $(a_0*\cdots*a_k)*a_{k+1} = a_0*(a_1*\cdots*a_{k+1})$, where the $*$'s in parentheses are evaluated from left to right.
This gives a one-parameter generalization of the usual associativity (i.e., $1$-associativity), which was also studied by Waldhauser~\cite[Theorem~2.3]{AlmostAssociative}.

On the other hand, we will see in Section~\ref{sec:nil-lr} that, for $e=k=\ell=1$, the ``if'' part of the $(*,n)$-relation \eqref{eq:mod-nil-lr} can be viewed as \demph{associativity at left depth $d$}, that is, $t\sim_* t'$ if $t$ can be obtained from $t'$ by a finite sequence of moves, each of which replaces the maximal subtree rooted at a node of left depth at least $d-1$ by another binary tree with the same number of leaves.

\subsection{Summary of the main results}
The two $1$-parameter generalizations of associativity given above justify the terminology ``$(k,\ell)$-associativity at depth $(d,e)$'' for the ``if'' part of \eqref{eq:mod-nil-lr}.
In this paper we focus on two special cases, $k=\ell=1$ and $e=\ell=1$, each giving a two-parameter generalization of the usual associativity with connections to many interesting integer sequences and combinatorial objects.

In Section~\ref{sec:nil-lr} we study the case $k=\ell=1$.
In this case the ``if'' part of \eqref{eq:mod-nil-lr} with $k=\ell=1$ can be viewed as \demph{associativity at left depth $d$ and right depth $e$}, which recovers the associativity at left depth $d$ when $e=1$.
Define $\defcolor{C^{d,e}_n} :=  C^{d,e}_{1,1,n}$ and $\defcolor{\tC^{d,e}_n} := \tC^{d,e}_{1,1,n}$.
We determine ${\tC^{d,e}_n}$ and enumerate $(*,n)$-classes with this size for arbitrary binary operations $*$ satisfying \eqref{eq:mod-nil-lr} with $k=\ell=1$.
We show that the cardinality of each $(*,n)$-class is a product of Catalan numbers in Corollary~\ref{cor:prod-cat}.
We also provide a recursive formula for the generating function $\defcolor{C^{d,e}(x)}$ of ${C^{d,e}_n}$. 
Then we give closed formulas for $C^{d,e}(x)$ and $C^{d,e}_n$ when $e=1,2$.
It turns out that $C^{d,1}(x)$ is a well-known continued fraction, and $\{C^{d,1}_n:d\ge1,n\ge0\}$ coincides with an array in OEIS~\cite[A080934]{OEIS}, which enumerates various families of objects, including
\begin{itemize}
\item
binary trees with $n+1$ leaves of left depth at most $d$ (by~Proposition~\ref{prop:MC}),
\item
plane trees with $n+1$ nodes of depth at most $d$ (de Bruijn, Knuth, and Rice~\cite{DBKR}),
\item
Dyck paths of length $2n$ with height at most $d$ (Flajolet~\cite{Flajolet} and Kreweras~\cite[page 38]{Kreweras}),
\item
permutations in $\SS_n$ avoiding $132$ and $123\cdots(d+1)$  (Krattenthaler~\cite{Krattenthaler} 
, Kitaev, Remmel, and Tiefenbruck~\cite{KRT12}), 
\item
ad-nilpotent ideals of the Borel subalgebra of the Lie algebra $\mathfrak{sl}_n(\mathbb C)$ of order at most $d-1$ (Andrews, Krattenthaler, Orsina, and Papi~\cite{AdNilIdeals}).
\end{itemize}
There are previously known closed formulas for $C^{d,1}_n$~\cite{AdNilIdeals,DBKR}, but our formula (Prop.~\ref{prop:d1}) is apparently different.
The number $C^{d,2}_n$ occurs in OEIS only for $d\le3$~\cite[A045623, A142586]{OEIS};
we find no result on $C^{d,2}$ for $d\ge4$ or $C^{d,3}_n$ for $d\ge 3$ in OEIS~\cite{OEIS}.

In Section~\ref{sec:mod-nil} we study the case $e=\ell=1$.
In this case the ``if'' part of \eqref{eq:mod-nil-lr} can be viewed as \demph{$k$-associativity of left depth $d$}, which recovers the $k$-associativity when $d=1$ and recovers the associativity of left depth $d$ when $k=1$.
We give a few families of combinatorial objects enumerated by $\defcolor{C^d_{k,n}} := C^{d,1}_{k,1,n}$, including binary trees, plane trees, and Dyck paths with certain constraints, and establish a recursive formula for the generating function $\defcolor{C^d_k(x)}$ of $C^d_{k,n}$.
Then we study $C^d_{k,n}$ when $d=1,2,3$ or $k=1,2$.
We have $C_1^d(x)=C^{d,1}(x)$ and $C_2^d(x) = C^{d+1,1}(x)$ for $d,n\ge0$.
The sequence $\{C_{3,n}^d\}$ has been studied by Barcucci, Del Lungo, Pergola, and Pinzani~\cite{CatPerm} in terms of pattern avoidance in permutations; see also~\cite[A005773, A054391--A054394]{OEIS} for $d=1,\ldots,5$.
There is a closed formula for $C_3^d(x)$ but no closed formula for $C_{3,n}^d$ given in~\cite{CatPerm}.
We provide a different formula for $C_3^d(x)$ and derive a closed formula for $C_{3,n}^d$ from it.
We also give closed formulas for the number $C^d_{k,n}$ when $d=2$, using Lagrange inversion and our ealier work on $C^1_{k,n}$~\cite{CatMod}.


Another two-parameter specialization of \eqref{eq:mod-nil-lr} can be obtained by taking $d=e=1$.
Let $\defcolor{C_{k,\ell,n}} := C^{1,1}_{k,\ell,n}$.
Computations suggest a conjecture: $C_{k,\ell,n} = C_{k+\ell-1,n}$ for all $k,\ell\ge1$ and $n\ge0$.  
One could also explore the case $d,e,k,\ell>1$ in the future.

\tableofcontents

\section{Associativity at depth $(d,e)$: the case $k=\ell=1$}\label{sec:nil-lr}

In this section we assume $*$ is a binary operation satisfying \eqref{eq:mod-nil-lr} with $k=\ell=1$ and $d,e\ge 1$.
This means, for $t,t'\in \T_n$, $t\sim_* t$ if and only if both $\ld(t) \sim^d \ld(t')$ and $\rd(t) \sim^{e} \rd(t')$.
We study $(*,n)$-classes and the nonassociativity measurements $\defcolor{C^{d,e}_n} := C_{*,n}$ and $\defcolor{\tC^{d,e}_n} := \tC_{*,n}$ of $*$ arising from these classes.

\subsection{Equivalence classes}
We first introduce some notation.
If a node in a binary tree has left depth $\ld\ge d-1$ and right depth $\rd\ge e-1$ then we say this node is \demph{$(d,e)$-contractible}, or simply \demph{contractible} if $d$ and $e$ are clear from the context.
We call a contractible node \demph{maximal} if its parent is not contractible.
One sees that a node with left depth $\ld$ and right depth $\rd$ is a maximal contractible node if and only if $\ld=d-1$ and $\rd\ge e-1$ when $v$ is the left child of its parent, or $\ld\ge d-1$ and $\rd=e-1$ when $v$ is the right child of its parent.

Let $t\in\T_n$ and assign each leaf weight one. 
For each maximal contractible node $v$, we contract its subtree to a single node and assign $v$ a weight equal to the number of leaves in this subtree.
Denote by $\defcolor{\phi(t)}$ the resulting weighted binary tree.
This gives a map $\phi:\T_n\to\T^{d,e}_n$ by $t\mapsto\phi(t)$, where $\defcolor{\T^{d,e}_n}$ is the set of all leaf-weighted binary trees such that
\begin{itemize}
\item
every contractible leaf is maximal and has a positive integer weight,
\item
every non-contractible leaf has weight one, and
\item 
the sum of leaf weights is $n+1$.
\end{itemize}

Conversely, let $\bar t\in\T^{d,e}_n$ have leaves $v_0,\ldots,v_r$ with weights $m_0,\ldots,m_r$ respectively.
We replace $v_i$ by an arbitrary binary tree $t_i$ with $m_i$ leaves for $i=0,\ldots,r$.
Write $\defcolor{\phi^{-1}(\bar t;t_0,\ldots,t_r)}$ for the resulting tree.

\begin{lemma}\label{lem:phi}
(i) We have a surjection $\phi:\T_n\twoheadrightarrow\T^{d,e}_n$. 

\noindent(ii) For each $\bar t\in\T^{d,e}_n$ whose leaves $v_0,\ldots,v_r$ are weighted $m_0,\ldots,m_r$, its fiber is
\[ \phi^{-1}(\bar t) = \{ \phi^{-1}(\bar t;t_0,\ldots,t_r) : t_i\in \T_{m_i-1}\}. \]

\noindent (iii) We have $t\sim_* t'$ whenever $\phi(t)=\phi(t')$.
\end{lemma}

\begin{proof}
We first prove the third statements, and the others quickly follow.
Let $\bar t\in\T^{d,e}_n$ with leaves $v_0,\ldots,v_r$ having weight $m_0,\ldots,m_r$.
As $\sim_*$ is an equivalence relation, so, to prove (iii) it suffices to show $t=\phi^{-1}(\bar t;t_0,\ldots,t_r) \sim_* t'=\phi^{-1}(\bar t;t'_0,\ldots,t'_r)$, where $t_j$ are $t'_j$ are distinct binary trees with $m_j$ leaves for some $j\in\{0,\ldots,r\}$ and $t_i$ and $t'_i$ are the same binary tree with $m_i$ leaves for all $i\ne j$.
Suppose the $j$th leaf $v_j$ has left depth $\ld$ and right depth $\rd$.
Since $v_j$ is contractible by the definition of $\T^{d,e}_n$, we have $\ld\ge d-1$ and $\rd\ge e-1$.
We also know $m_j\ge 3$ as $t_j\ne t'_j$, so in $t$ we have:
\begin{itemize}
\item
the first leaf of $t_j$ has left depth at least $\ld+1\ge d$ and right depth equal to $\rho$, 
\item
the last leaf of $t_j$ has left depth equal to $\ld$ and right depth at least $\rd+1\ge e$, and  \item
all other leaves of $t_j$ have left depths at least $\ld+1\ge d$ and right depths at least $\rd+1\ge e$.
\end{itemize}
Similarly, the leaves of $t'_j$ satisfy the same properties.
Thus $\ld(t) \sim^d \ld(t')$ and $\rd(t) \sim^{e} \rd(t')$.
This proves (iii).

The argument above implies that each descendant of a contractible node is contractible, so (ii) also holds.
Furthermore, each fiber $\phi^{-1}(\bar t)$ is nonempty as one can choose trees $t_i$ with $m_i$ leaves for each $i$.
This implies the surjectivity of $\phi$.
\end{proof}

\begin{theorem}\label{thm:nil-lr}
Let $*$ be a binary operation satisfying \eqref{eq:mod-nil-lr} with $k=\ell=1$.
Then the fibers of $\phi$ are precisely the $(*,n)$-classes.
\end{theorem}

\begin{proof}
Let $\bar s\in\T^{d,e}_n$ whose leaves $u_0,\ldots,u_a$ are weighted $m_0,\ldots,m_a$.
Let $\bar t\in\T^{d,e}_n$ whose leaves $v_0,\ldots,v_b$ are weighted $n_0,\ldots,n_b$.
Assume $\bar s$ and $\bar t$ are distinct.
Let $s=\phi^{-1}(\bar s; s_1,\ldots,s_a)\in \phi^{-1}(\bar s)$ and $t=\phi^{-1}(\bar t; t_1,\ldots,t_b)\in\phi^{-1}(\bar t)$. 
Assume for a contradiction that $s\sim_* t$.

Let $j$ be the smallest integer such that $m_j\ne n_j$, say $m_j<n_j$.
Then $m_i=n_i$ for all $i<j$.
Moreover, $n_j>1$ implies that $v_j$ is a maximal contractible node, i.e., $\ld(v_j)\ge d-1$, $\rd(v_j)\ge e-1$, and equality holds in at least one of these two inequalities.

The last leaf of $s_j$ has left depth equal to $\ld(u_j)$ in $s$ and the $m_j$th leaf of $t_j$ has left depth at least $\ld(v_j)+1\ge d$ in $t$.
Then $s\sim_*t$ implies $\ld(u_j)\ge d$. 
Since the parent of $u_j$ is not contractible, we have $\rd(u_j)<e-1$ if $u_j$ is a left child or $\rd(u_j)<e$ if $u_j$ is a right child.

One also sees that the first leaf of $s_j$ has right depth equal to $\rd(u_j)$ in $s$ and the first leaf of $t_j$ has right depth equal to $\rd(v_j)$ in $t$.
Combining this with $\rd(u_j)<e$ and $\rd(v_j)\ge e-1$ we have $\rd(u_j)=\rd(v_j)=e-1$.
Thus $u_j$ must be a right child.
This implies that $u_{j+1}$ has right depth at most $\rd(u_j)=e-1$ in $\bar s$.
Then the first leaf of $s_{j+1}$ has right depth at most $e-1$ in $s$.
On the other hand, the $(m_j+1)$th leaf of $t_j$ has right depth at least $\rd(v_j)+1=e$ in $t$.
This gives a contradiction.
\end{proof}

\begin{corollary}\label{cor:prod-cat}
The cardinality of each $(*,n)$-class is a product of Catalan numbers $C_{m_0-1}\cdots C_{m_r-1}$ with $m_0+\cdots+m_r = n+1$.
\end{corollary}

\begin{proof}
By Theorem~\ref{thm:nil-lr}, every $(*,n)$-class can be written as $\phi^{-1}(\bar t)$ for some $\bar t\in\T^{d,e}_n$.
If the leaves of $\bar t$ have weights $m_0,\ldots,m_r$ then $|\phi^{-1}(\bar t)| = C_{m_0-1}\cdots C_{m_r-1}$ by Lemma~\ref{lem:phi}. 
We have $m_0+\cdots+m_r = n+1$ by the definition of $\T^{d,e}_n$.
\end{proof}

\subsection{Nonassociativity measurements}

We first determine the largest size  $\tC^{d,e}_n$ of a $(*,n)$-class.

\begin{theorem}
Let $d,e\ge1$.
If $0\le n<d+e$ then $\tC^{d,e}_n=1$.
If $n\ge d+e$ then $\tC^{d,e}_n=C_{n+2-d-e}$ and the number of $(*,n)$-classes with this size is ${d+e-2\choose d-1}$.
\end{theorem}

\begin{proof}
It is well known that the Catalan sequence $\{C_n\}$ is \demph{log-convex}, i.e., $C_m C_n\le C_{m+1}C_{n-1}$ for $m\ge n\ge1$.
Hence for each $r\ge 0$ the largest result from products of the form $C_{m_0-1}\cdots C_{m_r-1}$ with $m_0+\ldots+m_r=n+1$ is $C_{n-r}$, which is attained when all but one of $m_0,\ldots,m_r$ equal one.

Now let $\bar t\in\T^{d,e}_n$ with leaves $v_0,\ldots,v_r$ weighted $m_0,\ldots,m_r$.
If $m_0=\cdots=m_r=1$ then $C_{m_0-1}\cdots C_{m_r-1}=1$.
Assume $m_i>1$ for some $i$. 
Then $v_i$ is contractible, i.e. $\ld(v_i)\ge d-1$ and $\rd(v_i)\ge e-1$.
Hence the unique path from the root of $\bar t$ to $v_i$ has length at least $d+e-2$.
This implies that $\bar t$ has at least $d+e-2$ internal nodes and thus $r\ge d+e-2$. 
It follows that $|\phi^{-1}(\bar t)|\le C_{n+2-d-e}$, in which equality holds only if $r=d+e-2$ and all but one of $m_0,\ldots,m_r$ equal one.
We have $n+1=m_0+\cdots+m_r\ge r+2 \ge d+e$, i.e., $n\ge d+e-1$.
If $n=d+e-1$ then $C_{n+2-d-e}=C_1=1$.
Thus we assume $n\ge d+e$ below.

It remains to show that there are precisely ${d+e-2\choose d-1}$ many trees in $\T^{d,e}_n$ with $d+e-1$ leaves, one having weight larger than one and all others having weight one.
Suppose that $\bar t$ is such a tree and let $v$ be its unique leaf with weight larger than one.
Then the unique path from the root of $\bar t$ down to $v$ has length $d+e-2$.
This path has $d-1$ left steps and $e-1$ right steps.
Thus there are precisely ${d+e-2\choose d-1}$ many choices for this path.
The entire tree $\bar t$ is determined by this path as all nodes not on this path must be leaves.
\end{proof}

Now we study the other nonassociativity measurement $C^{d,e}_n$ and the generating function
\[ \defcolor{C^{d,e}(x)} := \sum_{n\ge0} C^{d,e}_n x^{n+1}.\]
By symmetry we have $C^{d,e}(x) = C^{e,d}(x)$.
If $d$ or $e$ is zero then we treat it as one.
Theorem~\ref{thm:nil-lr} implies a recurrence relation for $C^{d,e}(x)$.

\begin{proposition}\label{prop:nil-lr-recurrence}
For $d,e\ge1$ we have 
\[ C^{d,e}(x) = x + C^{d-1,e}(x)C^{d,e-1}(x) \]
\end{proposition}

\begin{proof}
By Theorem~\ref{thm:nil-lr}, $C^{d,e}_n=|\T^{d,e}_n|$.
For any $\bar t\in\T^{d,e}_n$, let $\bar t_L$ and $\bar t_R$ denote the (maximal) weighted subtrees rooted at the left and right children of the root of $\bar t$.
Then $\bar t_L\in\T^{d-1,e}_m$ and $\bar t_R\in\T^{d,e-1}_{n-m-1}$ for some $m$.
Two trees $\bar s$ and $\bar t$ in $\T^{d,e}_n$ are equal if and only if $\bar s_L=\bar t_L$ and $\bar s_R = \bar t_R$.
The result follows.
\end{proof}

Proposition~\ref{prop:nil-lr-recurrence} allows us to compute $C^{d,e}(x)$ recursively from $C^{d,1}(x)$ and $C^{1,e}(x)$.
By the symmetry between $d$ and $e$, we just need to determine $C^{d,1}(x)$.
We will achieve that in the next subsection.

\subsection{Associativity of left depth $d$: The case $e=k=\ell=1$}\label{sec:nil}
We apply our previous results to the case $e=k=\ell=1$.
In this case the $(*,n)$-relation \eqref{eq:mod-nil-lr} can be regarded as \demph{associativity at left depth $d$}, since Theorem~\ref{thm:nil-lr} implies that in this case two trees $t,t'\in\T_n$ satisfy $t\sim_* t'$ if and only if $t$ can be obtained from $t'$ by a finite sequence of moves, each of which replaces the maximal subtree rooted at a node of left depth at least $d-1$ by another binary tree containing the same number of leaves.

Using Proposition~\ref{prop:nil-lr-recurrence} we can determine the number $\defcolor{C^d_n} := C^{d,1}_n$ of $(*,n)$-classes and the generating function $\defcolor{C^d(x)} := C^{d,1}(x)$.
To state a formula for $C^d(x)$, we need the \demph{Fibonacci polynomials} defined by $F_n(x) := F_{n-1}(x) - x F_{n-2}(x)$ for $n\ge2$ with $F_0(x):=0$ and $F_1(x):=1$.
For example, we have $F_2(x) = 1$, $F_3(x) = 1-x$, $F_4(x) = 1-2x$, $F_5(x) = 1-3x+x^2$, $F_6(x) = 1-4x+3x^2$, $F_7(x) = 1-5x+6x^2-x^3$, and so on.
For $n\ge1$ we have \cite[(8), (9), (10)]{DBKR}
\begin{align*}
F_n(x) &= \frac{1}{\sqrt{1-4x}} \left( \left(\frac{1+\sqrt{1-4x}}{2}\right)^n - \left( \frac{1-\sqrt{1-4x}}{2} \right)^n \right) \\
&= \sum_{0\le i\le(n-1)/2} {n-1-i\choose i}(-x)^i 
= \prod_{1\le j\le (n-1)/2} (1-4x\cos^2(j\pi/n)). 
\end{align*} 
Using the \demph{Chebyshev polynomial} $U_n(x) := 2xU_{n-1}(x)-U_{n-2}(x)$ with $U_0(x)=1$ and $U_1(x)=2x$, one has $\displaystyle F_{n+1}(x) = x^{\frac n2} U_{n}\left(\frac{1}{2\sqrt{x}} \right)$ for $n\ge0$.

\begin{corollary}[Kreweras~\cite{Kreweras}]\label{cor:nil}
For $d\ge1$ we have (with $C^0(x):=x$)
\[ C^{d}(x) = \frac{x}{1-C^{d-1}(x)}= \frac{ xF_{d+1}(x) }{ F_{d+2}(x)} .\]
\end{corollary}

\begin{proof}
Observe that $C^1(x) = x/(1-x)$.
By Proposition~\ref{prop:nil-lr-recurrence}, we have $C^{d}(x) = x + C^{d-1}(x)C^{d}(x)$ and thus $C^{d}(x) = x / (1-C^{d-1}(x))$.
By induction on $d$ we have $C^{d}(x) = xF_{d+1}(x)/F_{d+2}(x)$.
\end{proof}

In addition to work of Kreweras~\cite{Kreweras}, other papers have also discussed $C^d(x)$; see, e.g., Chow and West~\cite[Theorem~3.6]{ChowWest}, Krattenthaler~\cite[Theorem~2]{Krattenthaler}, and Mansour and Vainshtein~\cite[Theorem~1.1]{MansourVainshtein}.
Corollary~\ref{cor:nil} implies that $C^d(x)$ is a well-known continued fraction~\cite{DBKR,Flajolet}.
For example, we have
\[ C^{1}(x) = \frac x{1-x}, \quad 
C^{2}(x)  = \frac x { 1- \frac x {1-x} } = \frac{x(1-x)}{1-2x}, \quad
C^{3}(x)  = \frac x {1- \frac x { 1- \frac x {1-x} } } = \frac{x(1-2x)}{1-3x+x^2}, \quad \ldots.
\]
Hence $(C^d_n)_{d\ge1,n\ge0}$ coincides with an array in OEIS~\cite[A080934]{OEIS} and enumerates 
\begin{itemize}
\item
plane trees with $n+1$ nodes of depth at most $d$ (de Bruijn, Knuth, and Rice~\cite{DBKR}),
\item
Dyck paths of length $2n$ with height at most $d$ (Flajolet~\cite{Flajolet} and Kreweras~\cite[page 38]{Kreweras}),\item
permutations in $\SS_n$ avoiding $132$ and $123\cdots(d+1)$ (Krattenthaler~\cite{Krattenthaler}, Kitaev, Remmel, and Tiefenbruck~\cite{KRT12}), 
\item
ad-nilpotent ideals of the Borel subalgebra of the Lie algebra $\mathfrak{sl}_n(\mathbb C)$ of order at most $d-1$ (Andrews, Krattenthaler, Orsina, and Papi~\cite{AdNilIdeals}).
\end{itemize}
Note that plane trees with $n+1$ nodes of depth at most $d$ correspond to binary trees with $n+1$ leaves of left depth at most $d$.
The latter is a family of objects more relevant to our current work and can also be obtained from our later result Proposition~\ref{prop:MC} by setting $k=1$. 

There are many known closed formulas for the number $C^d_n$, such as
\begin{align*}
C^d_n &= \sum_{i\in\mathbb Z} \frac{2i(d+2)+1}{2n+1} \binom{2n+1}{n-i(d+2)} \qquad \text{\cite[Thm. 4.5]{AdNilIdeals}} \\ 
& = \det \left[ \binom{i-\max\{-1,j-d\}}{j-i+1} \right]_{i,j=1}^{n-1} 
\qquad \text{\cite[Thm. 4.5]{AdNilIdeals}} \\
&= \sum_{0=i_0\le i_1\le \cdots \le i_{d-1}\le i_{d}=n} \prod_{0\le j\le d-2} \binom{i_{j+2}-i_j-1}{i_{j+1}-i_j} 
\qquad \text{\cite[Cor. 4.3]{AdNilIdeals}} \\
&= \frac{2^{2n+1}}{d+2} \sum_{1\le j\le d+1} \sin^2 \frac{j\pi}{d+2} \cos^{2n}\frac{j\pi}{d+2}.  
\qquad \text{\cite[(14)]{DBKR}} 
\end{align*}
When $d$ is small, the number $C^d_n$ satisfies simpler formulas: one has $C^2_n=2^{n-1}$, $C^{3}_n = F_{2n-1}$, 
and $C^{4}_n = \frac12 (1+3^{n-1})$ for $n\ge1$~\cite{AdNilIdeals,KRT12}.

Now we derive a closed formula for $C^d_n$ from the generating function $C^d(x)$, which seems different from other known formulas for this number.
We write $\alpha\models n$ if $\alpha$ is a \demph{composition} of $n$, i.e., if $\alpha=(\alpha_1,\ldots,\alpha_\ell)$ is a sequence of positive integer such that $\alpha_1+\cdots+\alpha_\ell=n$.
We also define $\ell(\alpha):=\ell$ and $\max(\alpha) := \max\{\alpha_1,\ldots,\alpha_\ell\}$.

\begin{proposition}\label{prop:d1}
For $n,d\ge1$ we have
\[ C^d_n =  \sum_{ \substack{\alpha\models n \\ \max(\alpha) \le (d+1)/2 } } (-1)^{n-\ell(\alpha)} \binom{d-\alpha_1}{\alpha_1-1} \prod_{2\le r\le \ell(\alpha)} \binom{d+1-\alpha_i}{\alpha_i} \]
\end{proposition}

\begin{proof}

For $d\ge1$ we have
\begin{align*}
C^{d}(x) & =  x + \frac{x^2 F_d(x)}{ F_{d+2}(x)} 
= x + \frac{\sum_{0\le i\le(d-1)/2} {d-1-i\choose i} (-x)^{i+2}} 
{ 1+ \sum_{1\le i\le(d+1)/2} {d+1-i\choose i} (-x)^i } \\
&= x + \sum_{0\le i\le(d-1)/2} {d-1-i\choose i} (-x)^{i+2} \sum_{j\ge0} \left(\sum_{1\le i\le(d+1)/2} {d+1-i\choose i} (-1)^{i-1} x^i \right)^j
\end{align*}
Extracting the coefficient of $x^{n+1}$ from $C^d(x)$ gives
\begin{align*}
C^d_n & = \sum_{0\le i\le (d-1)/2} (-1)^i {d-1-i\choose i} \sum_{j\ge0} \sum_{ \substack{ 1\le i_1,\ldots,i_j\le (d+1)/2 \\ i_1+\cdots+i_j = n-1-i } } \prod_{1\le r\le j} (-1)^{i_r-1} \binom{d+1-i_r}{i_r} \\
&=  \sum_{j\ge0} (-1)^{n-1-j} \sum_{ \substack{ 1\le i_0,i_1,\ldots,i_j\le (d+1)/2 \\ i_0+i_1+\cdots+i_j = n} } \binom{d-i_0}{i_0-1} \prod_{1\le r\le j} \binom{d+1-i_r}{i_r} .
\end{align*}
Viewing $(i_0,\ldots,i_j)$ as a composition $\alpha$ of $n$ gives the desired formula. 
\end{proof}

Next, we give a new interpretation of the number $C^d_n$, which is very similar to the one obtained by Andrews, Krattenthaler, Orsina, and Papi~\cite{AdNilIdeals}.
We do not know any quick way to convert from one to the other.

Let $\mathcal U_n$ be the algebra of $n$-by-$n$ upper triangular matrices over a field $\FF$, with the usual matrix product. 
Using column operations one can write a (two-sided) ideal $I$ of $\mathcal U_n$ as an upper triangular matrix $[a_{ij}]_{1\le i,j\le n}$ with $a_{ij}\in\{0,*\}$ such that $a_{ij}=*$ implies $a_{ij'}=*$ for all $j'\ge j$ and $a_{i'j}=*$ for all $i'\le i$.
The elements of $I$ are all matrices in $\mathcal U_n$ whose $(i,j)$-entry is arbitrary if $a_{ij}=*$ or zero if $a_{ij}=0$. 
The ideal $I = [a_{ij}]_{1\le i,j\le n}$ is nilpotent if and only if $a_{ii}=0$ for all $i\in[n]$.
Thus the stars in the matrix form of a nilpotent ideal $I$ give a partition inside the staircase $(n-1,\ldots,1,0)$. 
It follows that the number of nilpotent ideals of $\mathcal U_n$ is the Catalan number $C_n$.

The \demph{order} of a nilpotent ideal $I$ is $\inf\{d:I^d=0\}$.
Observe that an ideal $I$ of $\mathcal U_n$ is commutative if and only if $I^2=0$.
Shapiro~\cite{Shapiro} showed that the number commutative ideals of $\mathcal U_n$ is $2^{n-1}$.
We generalize this result below, using the number $C^d_n$. 

\begin{proposition}\label{prop:NumberOfIdeals}
For $d\ge1$, nilpotent ideals of order at most $d$ in $\mathcal U_n$ are enumerated by $C^d_n$.
\end{proposition}

\begin{proof}
\footnote{We are grateful to Brendon Rhoades for pointing out this proof and allowing us to include it here.}
A nilpotent ideal $I$ of $\mathcal U_n$ is determined by the lower boundary $D$ of the stars in its matrix form, which can be identified as a Dyck path of length $2n$.
We say a sequence $(i_1,i_2,\ldots,i_d)$ of $d$ integers between $1$ and $n$ is \demph{$D$-admissible} if the matrix entry $(i_j,i_{j+1})$ lies to the northeast of $D$ for all $j=1,2,\ldots,d-1$.
For any positive integer $m$, we have $I^m\ne0$ if and only if there exists a $D$-admissible sequence $(i_1,i_2,\ldots,i_{m+1})$.
Thus the order of $I$ is the largest integer $d$ such that there exists a $D$-admissible sequence $(i_1,\ldots,i_d)$.

Next, we construct the \demph{bounce path} of $D$ by starting from the northwest corner, going east until hitting a south step of $D$, then turning south and bouncing off the main diagonal to proceed east, and repeating this process all the way to the southeast corner.
The bounce path must be of the form $E^{b_1}S^{b_1}E^{b_2}S^{b_2}\cdots E^{b_k}S^{b_k}$, where $E$ is an east step, $S$ is a south step, and $(b_1,b_2,\ldots,b_k)$ is a composition of $n$.
We say this bounce path has $k$ \demph{parts}.

For $j=1,\ldots,k-1$, the matrix entry $(a_j,a_{j+1})$, defined by $a_1=1$ and $a_{j+1}=a_j+b_j$, is immediately to the right of the first south step of the segment $S^{b_j}$, which also lies to the northeast of $D$. 
Thus $(a_1,\ldots,a_k)$ is $D$-admissible.
This shows that $d\ge k$.

On the other hand, the longest $D$-admissible sequence $(i_1,\ldots,i_d)$ must satisfy $i_1=1=a_1$; 
otherwise $(1,i_1,\ldots,i_d)$ would be an even longer $D$-admissible sequence. 
Let $m$ be the largest integer such that $a_j=i_j$ for $j=1,2,\ldots,m$.
If $m<k$ then $a_{m+1}<i_{m+1}$ and $(a_1,\ldots,a_m,a_{m+1},i_{m+2},\ldots,i_d)$ is also $D$-admissible.
Repeating this process gives a $D$-admissible sequence $(a_1,a_2,\ldots,a_k,i_{k+1},\ldots,i_d)$.
But there is no entry on the $a_k$-th row to the northeast of $D$. 
This implies $d=k$.

There exists a bijection between Dyck paths of length $2n$ with exactly $h$ parts in their bounce paths and Dyck paths of length $2n$ with height $h$; see, e.g., Haglund~\cite[Theorem 3.15 and Remark 3.16]{Haglund}.
Thus nilpotent ideals of order at most $d$ in $\mathcal U_n$ are in bijection with Dyck paths of length $2n$ with height at most $d$; the latter family is known to be enumerated by $C^d_n$
~\cite[page 38]{Kreweras}.
This establishes the result.
\end{proof}

\subsection{Associativity at depth $(d,2)$}
Now we give closed formulas for $C^{d,2}(x)$ and $C^{d,2}_n$.

\begin{proposition}\label{prop:d2}
For $d\ge2$ we have 
\[ C^{d,2}(x) 
= C^d(x) + \frac{x^{d+2}}{(1-2x) F_{d+2}(x)}. \]
Consequently, for $n,d\ge2$ we have
\[ C^{d,2}_n = C^d_n + \sum_{1\le i\le n-d} 2^{i-1} \sum_{ \substack{ \alpha\models n-d-i \\ \max(\alpha) \le (d+1)/2 }} (-1)^{n-d-i-\ell(\alpha) } \prod_{1\le j\le \ell(\alpha)} \binom{d+1-\alpha_j}{\alpha_j} .\]
\end{proposition}

\begin{proof}
By Proposition~\ref{prop:nil-lr-recurrence} and Corollary~\ref{cor:nil}, we can verify the formula for $d=2$, 
\[ C^{2,2}(x) = x + \frac{x^2(1-x)^2}{(1-2x)^2} = \frac{x(1-2x)(1-x)+x^4}{(1-2x)^2}. \]
For $d\ge3$ we obtain the desired formula for $C^{d,2}(x)$ by induction on $d$, Proposition~\ref{prop:nil-lr-recurrence}, and Corollary~\ref{cor:nil}.
It follows that
\begin{align*}
C^{d,2}(x) 
&= C^d(x) + \frac{\sum_{i\ge0} 2^i x^{d+i+2} }{ 1 + \sum_{1\le i\le (d+1)/2} \binom{d+1-i}{i} (-x)^i } \\
&= C^d(x) + \sum_{i\ge0} 2^i x^{d+i+2} \sum_{j\ge0} \left( \sum_{1\le i\le (d+1)/2} \binom{d+1-i}{i} (-1)^{i-1} x^i \right)^j .
\end{align*}
Hence
\begin{align*}
C^{d,2}_n &= C^d_n + \sum_{i\ge0} 2^{i} \sum_{j\ge0}  (-1)^{n-1-d-i-j} \sum_{ \substack{ 1\le i_1,\ldots,i_j\le (d+1)/2 \\ i_1+\cdots+i_j = n-1-d-i } } \prod_{1\le r\le j} \binom{d+1-i_r}{i_r} .
\end{align*}
This implies the desired formula for $C^{d,2}_n$.
\end{proof}

Using Proposition~\ref{prop:d2} we find $C^{2,2}_n$ and $C^{3,2}_n$ in OEIS.
The sequence $\{C^{2,2}_n:n\ge0\}$ is the binomial transformation of $1,1,2,2,3,3,\ldots$, has a simple formula $C^{2,2}_n = (n+2)2^{n-3}$ for $n\ge2$, and enumerates a few families of objects, such as copies of $r$ in all compositions of $n+r$ for any positive integer $r$, weak compositions of $n-1$ with exactly one zero, triangulations of a regular $(n+3)$-gon in which each triangle contains at least one side of the polygon, and so on~\cite[A045623]{OEIS}.
The sequence $(C^{3,2}_n)_{n\ge0}$ is the binomial transformation of $(\lfloor (\frac{1+\sqrt5}2)^n \rfloor )_{n\ge0}$ and satisfies the formula $C^{3,2}_n = \left(\frac{1+\sqrt5}{2}\right)^{2n-2} + \left(\frac{1-\sqrt5}{2}\right)^{2n-2} - 2^{n-2}$ for $n\ge2$~\cite[A142586]{OEIS}. 
We do not see $C^{4,2}_n$ in OEIS, but it also satisfies a simple formula.

\begin{proposition}
For $n\ge3$ we have $C^{4,2}_n = 1+ 5\cdot 3^{n-3} - 2^{n-3}$.
\end{proposition}

\begin{proof}
By Proposition~\ref{prop:d2} we have
\begin{align*}
C^{4,2}(x) & = x + \frac{x^2(1-2x)}{1-4x+3x^2} + \frac{x^6}{(1-2x)(1-4x+3x^2)} \\
&= x + \left( \frac{1}{1-x}+\frac{1}{1-3x}\right) \frac{x^2}{2} + \left( \frac{1}{1-x} - \frac{8}{1-2x} + \frac{9}{1-3x} \right) \frac{x^6}{2} .
\end{align*}
This implies the desired formula for $C^{4,2}_n$. 
\end{proof}


We find no result on $C^{d,3}_n$ for $d\ge 3$ in OEIS.


\section{$k$-associativity of left depth $d$: the case $e=\ell=1$}\label{sec:mod-nil}

In this section we assume $*$ is a binary operation satisfying \eqref{eq:mod-nil-lr} with $e=\ell=1$, i.e., $t\sim_* t'$ if and only if $\ld(t) \sim_k^d \ld(t')$ for all $t,t'\in\T_n$. 
For $d,k\ge1$, we study $\defcolor{C^d_{k,n}} := C_{*,n}$ and $\defcolor{\tC^d_{k,n}} := \tC_{*,n}$.

\subsection{Equivalence classes}
We first study $(*,n)$-classes using rotations on binary trees.
Given a node $v$ in a binary tree $t$, a \demph{subtree rooted at $v$} is a subtree of $t$ whose root is $v$, and the \demph{maximal subtree rooted at $v$} is the subtree consisting of all descendants of $v$, including $v$ itself.
Given two binary trees $s$ and $t$, we say \demph{$t$ contains $s$ at left depth $d$} if $t$ contains $s$ as a subtree rooted at a node of left depth $d$ and \demph{$t$ avoids $s$ at left depth $d$} otherwise.

Given binary trees $s$ and $t$, write $\defcolor{s\wedge t}$ for the binary tree whose root has left and right maximal subtrees $s$ and $t$, respectively.
There is a natural bijection between the set of parenthesizations of $x_0*\cdots *x_n$ and the set $\defcolor{\T_n}$ of binary trees with $n$ internal nodes (i.e., with $n+1$ leaves) by replacing each $x_i$ by a leaf labeled $i$ and replacing each $*$ by $\wedge$.

Let $t$ be a binary tree and $v$ a node.
Suppose the maximal subtree of $t$ rooted at $v$ can be written as $t_0\wedge\cdots\wedge t_{k+1}$, where $t_0,\ldots,t_{k+1}$ are binary trees and the operations $\wedge$ are performed left-to-right.
Replacing this subtree with $t_0\wedge(t_1\wedge\cdots\wedge t_{k+1})$ in $t$ gives another binary tree $t'$.
We call the operation $t\mapsto t'$ a \demph{right $k$-rotation at $v$}, and call the inverse operation $t'\mapsto t$ a \demph{left $k$-rotation at $v$}.  We illustrate a right $2$-rotation below.
\[
\Tree [.  [. [. \rule{8pt}{0pt} \rule{8pt}{0pt} ] \rule{8pt}{0pt} ] \rule{8pt}{0pt} ]\longrightarrow
\Tree [. \rule{8pt}{0pt} [. [. \rule{8pt}{0pt} \rule{8pt}{0pt} ] \rule{8pt}{0pt} ] ]
\]

\begin{lemma}\label{lem:rotation}
A left or right $k$-rotation at a node $v$ in a binary tree $t$ produces another binary tree $t'$ satisfying $t\sim_k^d t'$ if and only if the left depth of $v$ in $t$ is at least $d-1$. 
\end{lemma}

\begin{proof}
A right $k$-rotation $t\mapsto t'$ at $v$ replaces the maximal subtree $t_0\wedge\cdots\wedge t_{k+1}$ of $t$ rooted at $v$ with $t_0\wedge(t_1\wedge\cdots\wedge t_{k+1})$.
This corresponds to a change in left depth by subtracting $k$ from the left depth of each leaf of $t_1$ and leaving the left depths of all other leaves of $t$ invariant. 
If $\ld(v)$ is the left depth of $v$ in $t$, then the rightmost leaf of $t_1$ has left depth $\ld(v)+1+k$ in $t$, which is the smallest among all leaves of $t$, and has left depth $\ld(v)+1$ in $t'$.
Hence $t \sim_k^d t'$ if and only if $\ld(v)\ge d-1$.
For a left $k$-rotation the proof is similar. 
\end{proof}

If $s\in \T_n$ can be obtained from $t\in \T_n$ by finitely many left $k$-rotations at nodes of left depths at least $d-1$ then we say $s\le_k^d t$.
We call this partial order on $\T_n$ the \demph{${d\choose k}$-order}.
For $d=1$, this is the $k$-associative order, and when $d=k=1$ this further specializes to the Tamari order.
For details on these orders, including Hasse diagrams of the posets given by the ${1\choose 1}$-order and ${1\choose 2}$-order on $\T_4$, see~\cite{CatMod}.  A binary tree minimal or maximal under the ${d\choose k}$-order is called \demph{${d\choose k}$-minimal} or \demph{${d\choose k}$-maximal}.

\begin{example}
The poset in Figure~\ref{fig:22order} shows a connected component of the  ${2\choose 2}$-order on $\T_6$.  Observe that this component has a unique ${2\choose 2}$-minimal tree and two ${2\choose 2}$-maximal trees.
\end{example}

\begin{figure}[h]
\[
 \includegraphics[width=.31\textwidth]{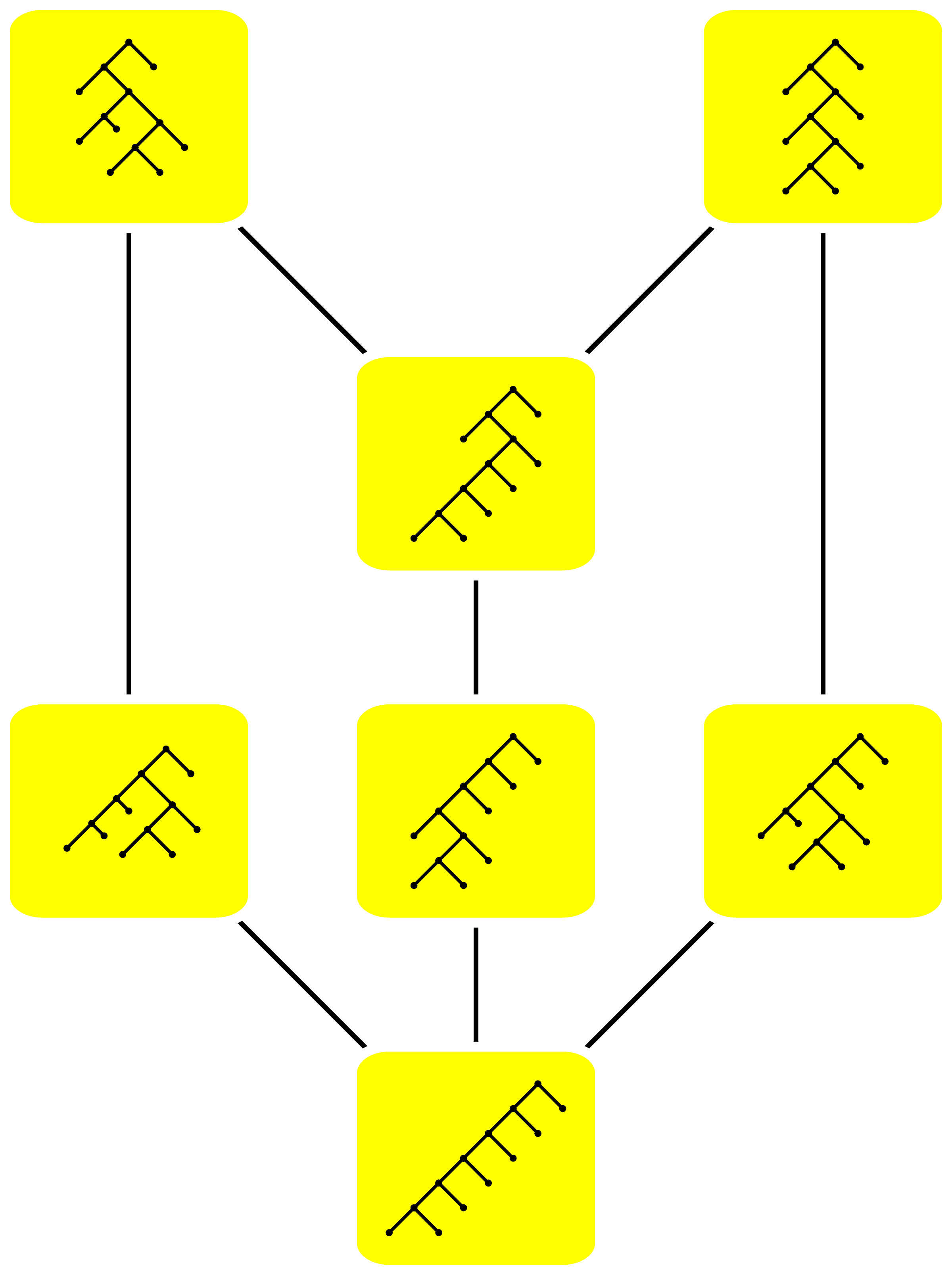}
\]
\caption{${2\choose 2}$-order on $\T_6$}\label{fig:22order}
\end{figure}

For each $k\ge1$ we define $\defcolor{\comb_k} := t_0\wedge\cdots\wedge t_k$ and $\defcolor{\comb_k^1} := t_0\wedge\comb_k$ where $t_0=\cdots=t_k \in \T_0$ (for a picture, see our previous work~\cite[Fig.~2.6]{CatMod}). 
The following result is straightforward.

\begin{proposition}\label{prop:comb}
Let $t$ be a binary tree. For $d,k\ge1$ we have

\noindent(i) $t$ is ${d\choose k}$-minimal if and only if it avoids $\comb_k^1$ at any left depth at least $d-1$, and

\noindent(ii) $t$ is ${d\choose k}$-maximal if and only if it avoids $\comb_{k+1}$ at any left depth at least $d-1$.
\end{proposition}


Let $t\in\T_n$.
The \demph{left border} of $t$ is the set of all nodes with right depth zero in $t$. 
Let $r_1,r_2,\ldots,r_h$ be the nodes on the left border of $t$ so that $r_1$ is the root and $r_{i+1}$ is the left child of $r_i$ for all $i=1,\ldots,h-1$.
For $2\le i\le h$ we define a tree $\defcolor{t^{+i}}\in\mathcal T_{n+1}$ by first cutting $t$ at $r_i$ to get the maximal subtree $s$ of $t$ rooted at $r_i$ and another subtree $s'$ of $t$ with leftmost leaf $r_i$, and then inserting the tree $u\in\mathcal T_1$, identifying the root of $u$ with the leaf $r_i$ in $s'$ and the right leaf of $u$ with the root $r_i$ of $s$.
The tree $\defcolor{t^{+1}}$ is obtained analogously by attaching $t$ to the right leaf of the tree $u\in\mathcal T_1$.
By construction, $\ld(t^{+i})=(i,\ld_0(t),\dotsc,\ld_n(t))$ for all $i=1,2,\ldots,h$.
For example, the left picture below is a tree $t\in\mathcal T_8$ with $\ld(t) = (5,4,3,3,3,2,2,1,0)$ and the right picture below is the tree $t^{+3}\in\mathcal T_9$ with $\ld(t^{+3}) = (3,5,4,3,3,3,2,2,1,0)$.

{\vskip10pt\scriptsize
 \Tree[. [. [. [. [. {\nola} {\nola} ] {\nola} ] [. {\nola} [. {\nola} {\nola} ] ] ] [. {\nola} {\nola} ] ] {\nola} ] 
 \Tree[. [. [. {\nola} [. [. [. {\nola} {\nola} ] {\nola} ] [. {\nola} [. {\nola} {\nola} ] ] ] ] !\qsetw{1in} [. {\nola} {\nola} ] ] {\nola} ] }

Conversely, write $\defcolor{t_-}$ for the binary tree in $\T_{n-1}$ obtained from $t\in\T_n$ by contracting the leftmost leaf, its sibling, and their parent to a single node.
One sees that $\ld(t^-) = (\ld_1,\ldots,\ld_n)$.
Thus $s\sim_k^d t$ implies $s^- \sim_k^d t^-$. 
Furthermore, if $\ld_0(t)=i$ then $t_-^{+i}:=(t_-)^{+i}=t$.

\begin{proposition}\label{prop:minimal}
Assume $*$ satisfies \eqref{eq:mod-nil-lr} with $e=\ell=1$. 
Then each $(*,n)$-class has a unique ${d\choose k}$-minimal element.
\end{proposition}

\begin{proof}
Since $\T_n$ is a finite set, each $(*,n)$-class must contain a ${d\choose k}$-minimal element.
We prove by induction on $n$ that each $(*,n)$-class has only one ${d\choose k}$-minimal element.
Let $s$ and $t$ are distinct binary trees in $\T_n$ such that $s\sim_k^d t$. 
We need to show that either $s$ or $t$ is not ${d\choose k}$-minimal, i.e., contains $\comb_k^1$ at left depth at least $d-1$.

First assume $s_-\ne t_-$. 
Since $s\sim_k^d t$ implies $s_- \sim_k^d t_-$, it follows from the induction hypothesis that, either $s_-$ or $t_-$, say the former, contains $\comb_k^1$ at left depth at least $d-1$.
Hence $s$ also contains $\comb_k^1$ at left depth at least $d-1$.

Next assume $s_-=t_-$.
This together with $s\ne t$ implies that $i=\ld_0(s)$ and $j=\ld_0(t)$ are distinct.
Assume $i<j$, without loss of generality.
Then $s=s_-^{+i}=t_-^{+i}$ contains $\comb_{j-i}^1$ at left depth $i-1$. 
Since $s\sim_k^d t=t_-^{+j}$, we have $i\ge d$ and $j=i+km$ for some $m>0$.
Thus $s$ contains $\comb_k^1$ at left depth $i-1\ge d-1$.
\end{proof}

Connected components of the Hasse diagram of the ${d\choose k}$-order on $\T_n$ are called \demph{${d\choose k}$-components}.

\begin{theorem}\label{thm:mod-nil}
Let $d,k\ge1$ and $n\ge0$.
Assume $*$ is a binary operation satisfying \eqref{eq:mod-nil-lr} with $e=\ell=1$.
Then the $(*,n)$-classes are precisely the ${d\choose k}$-components of $\T_n$.
\end{theorem}

\begin{proof}
By Lemma~\ref{lem:rotation}, each ${d\choose k}$-component of $\T_n$ is contained in some ${d\choose k}$-class.
We have the quality holds by Proposition~\ref{prop:minimal}.
\end{proof}

A subpath $L'$ of a lattice path $L$ is \demph{at height} $h$ if the initial point of $L'$ has height $h$.
We say $L$ \demph{avoids} $L'$ \demph{at height} $h$ is $L$ contains no subpath $L'$ at height $h$.

\begin{proposition}\label{prop:objects}
For $k,d\ge1$ and $n\ge0$, the number $C_{k,n}^d$ enumerates 
\begin{enumerate}
\item
binary trees with $n$ internal nodes avoiding $\comb_k^1$ at any left depth at least $d-1$,
\item
plane trees with multi-degree $(d_0,\ldots,d_n)$ satisfying $d_0+\cdots+d_{i-1}-i \ge d \Rightarrow d_i<k$, $\forall i\in[n]$,
\item
Dyck paths of length $2n$ avoiding $DU^k$ at height at least $d$,
\end{enumerate}
\end{proposition}

\begin{proof}
By Proposition~\ref{prop:minimal} and Theorem~\ref{thm:mod-nil}, $C_{k,n}^d$ enumerates the ${d\choose k}$-minimal elements in $\T_n$. 
Combining this with Proposition~\ref{prop:comb} establishes (1).

A binary tree $t$ with $n+1$ leaves corresponds to a plane tree $T$ with $n+1$ nodes by contracting northeast-southwest edges in $t$. 
By~\cite[Proposition~2.10]{CatMod}, the relation between the left depth $\ld(t) = (\ld_0,\ldots,\ld_n)$ and the multi-degree $d(T) = (d_0,\ldots,d_n)$ is given by 
\[ \ld_i = d_0 + \cdots + d_i - i ,\quad\forall i\in\{0,1,\ldots,n\}.\]
A left $k$-rotation at a node $v$ in $t$ corresponds to an up $k$-slide at a node of degree $d_i\ge k$ in $T$ for some $i\in[n]$.
One can check that the left depth of $v$ equals $d_0+\cdots+d_{i-1}-i$.
Thus $t$ is ${d\choose k}$-minimal if and only if $d_0+\cdots+d_{i-1}-i \ge d \Rightarrow d_i<k$ for all $i\in[n]$.
This implies (2).

Next, a plane tree with multi-degree $(d_0,\ldots,d_n)$ corresponds to a Dyck path $U^{d_0}DU^{d_1}\cdots DU^{d_n}$.
The height of the initial point of the $i$th down-step in this Dyck path is $d_0+\cdots+d_{i-1}-i+1$.
Thus (3) follows from (2).
\end{proof}

\begin{remark}\label{rem:LimitCkd}
For any fixed $n$ and $k$, the limit of $C_{k,n}^d$ as $d\to\infty$ is the Catalan number $C_{n}$ since the constraints in Proposition~\ref{prop:objects} are redundant if $d$ is large enough.
\end{remark}

For $k\ge1$ we define $\defcolor{M_{k-1,n}^d}$ to be the number of binary trees in $\T_n$ avoiding $\comb_{k}$ at any left depth at least $d-1$. 
By Proposition~\ref{prop:comb}, $M_{k,n}^d$ counts ${d\choose k}$-maximal elements in $\T_n$.
The number $M_{k,n} := M_{k,n}^1$ is called a \demph{generalized Motzkin number} in our earlier work~\cite{CatMod} and also studied by Tak\'acs~\cite{Takacs}.


\subsection{Generating functions}

For $d,k\ge1$ we define 
\[ C_k^d(x) := \sum_{n\ge0} C_{k,n}^d x^{n+1} \qand
M_{k-1}^d(x) := \sum_{n\ge0} M_{k-1,n}^d x^{n+1}. \]
We also set $C_k^0(x) := M_{k-1}^1(x)$.
To study these generating functions we need the following Lagrange inversion formula.

\begin{theorem}[{Stanley~\cite[Theorem~5.4.2]{EC2}}]
Suppose that $A(x)$ and $B(x)$ are formal power series such that $A(0)=B(0)=0$ and $A(B(x))=x$. Let $n,\ell\in\ZZ$. Then 
\[ n[x^n]B(x)^\ell = \ell[x^{n-\ell}] (x/A(x))^n.\]
\end{theorem}

Also recall the following binomial expansion for $m\ge0$:
\[ (1-x)^{-m} = \sum_{i\ge0} {m+i-1\choose i} x^i. \]

\begin{proposition}\label{prop:Ckd}
For $m,n,d\ge0$ and $k\ge1$ we have 
\[ C_k^{d+1}(x) = x \Big/ \left( 1-C_k^d(x) \right) \qand \]
\[ [x^{n+m}] C_k^{d+1}(x)^m = \sum_{0\le i\le n} { m+i-1 \choose i} [x^n] C_k^d(x)^i. \]
\end{proposition}

\begin{proof}
For $d=0$ the first equation follows from~\cite[(6)]{CatMod}.
Assume $d\ge1$. Let $t\in\T_n$ with $n\ge1$. 
Denote by $t_L$ and $t_R$ the maximal subtrees rooted at the left and right children of the root of $t$.
By Proposition~\ref{prop:comb}, $t$ is ${d+1\choose k}$-minimal if and only if $t_L$ is ${d\choose k}$-minimal and $t_R$ is ${d+1\choose k}$-minimal. 
Combining this with Proposition~\ref{prop:objects} we have
\[ C_k^{d+1}(x) =  x + C_k^d(x) C_k^{d+1}(x)\]
This implies the first equation.
Applying the binomial expansion gives the second equation.
\end{proof}

\begin{proposition}\label{prop:MC}
For $n,d\ge0$ and $k\ge1$ we have $M_{k-1}^{d+1}(x) = C_k^d(x)$ and $M_{k-1,n}^{d+1} = C_{k,n}^d$.\end{proposition}

\begin{proof}
\footnote{
It would be interesting to have a bijective proof for this result.}
The result holds for $d=0$ by definition. 
Similarly to the proof of Proposition~\ref{prop:Ckd}, we have
\[ M_k^{d+1}(x) =  x + M_k^d(x) M_k^{d+1}(x), \quad\forall\, d,k\ge0.\]
Hence $M_k^{d+1}(x) = x/(1-M_k^d(x))$.
The result then follows from induction on $d$.
\end{proof}
 
\begin{corollary}\label{cor:interlace}
For $d,k\ge1$ and $n\ge0$ we have $M_{k-1,n}^d \le C_{k,n}^d\le M_{k,n}^d$.
\end{corollary}

\begin{proof}
The first inequality follows from Proposition~\ref{prop:comb} and the second from Proposition~\ref{prop:minimal}.
\end{proof}

Combining Proposition~\ref{prop:MC} and Corollary~\ref{cor:interlace} we have the following diagram for $d,k\ge1$.
\[ \xymatrix @R=7pt @C=1pt { 
\cdots & & \cdots \ar@{=}[d] & & \cdots & & \cdots \ar@{=}[d] & & \cdots & & \cdots \ar@{=}[d] & & \cdots \\
\cdots & \le & M_{k-1,n}^d & \le & C_{k,n}^d \ar@{=}[d] & \le & M_{k,n}^d & \le & C_{k+1,n}^d \ar@{=}[d] & \le & M_{k+1,n}^d & \le & C_{k+2,n}^d & \le & \cdots \\
& & \cdots & \le & M_{k-1,n}^{d+1} & \le & C_{k,n}^{d+1} & \le & M_{k,n}^{d+1} & \le & C_{k+1,n}^{d+1} & \le & M_{k+1,n}^{d+1} \ar@{=}[u] & \le & C_{k+2,n}^{d+1} & \le & \cdots \\
& & & & \cdots & & \cdots \ar@{=}[u] & & \cdots & & \cdots \ar@{=}[u] & & \cdots & & \cdots \ar@{=}[u] & & \cdots \\
} \]

\begin{proposition}
For $d\ge1$ and $n\ge0$ we have $M_1^d(x) = C_1^d(x)$ and $M_{1,n}^d = C_{1,n}^d$.
\end{proposition}

\begin{proof}
We have $M_1^{1}(x) = x/(1-x) = C_1^{1}(x)$. 
The result then follows from induction on $d$, using the recurrence relations in Proposition~\ref{prop:Ckd} and the proof of Proposition~\ref{prop:MC}.
\end{proof}

\subsection{The case $k=3$ and $e=\ell=1$}\label{sec:k=3}

We already studied $C_{1,n}^d = C^d_n = C^{d,1}_n$ in Section~\ref{sec:nil}.
Now we determine $C_{2,n}^d$.

\begin{proposition}\label{prop:k=2}
For $d,n\ge0$ we have $C_2^d(x) = C_1^{d+1}(x)$ and $C_{2,n}^d=C_{1,n}^{d+1}$.
\end{proposition}

\begin{proof}
By definition, $C_{2}^{0}(x) = x/(1-x) =C_1^{1}(x)$.
By Proposition~\ref{prop:Ckd}, we have the same recurrence relation $C_k^{d+1}(x) = x/(1-C_k^d(x))$ for all $k\ge1$.
The result follows from induction on $d$.
\end{proof}

Next, we study $C_{3,n}^d$. 
By our earlier work~\cite{CatMod}, $C_{3,n}^{0}$ is the Motzkin number~\cite[A001006]{OEIS}, which has many closed formulas, and its generating function is 
\[ C_3^{0}(x) = \frac{1 - x - \sqrt{1-2x-3x^2} }{2x} = \frac x{1-x-\frac{x^2}{1-x-\frac{x^2}{\cdots} } }.\]
As a warm-up, we derive some closed formulas for $C_{3,n}^0$, which are probably well known, from the generating function $C_3^0(x)$. 
We have
\[ \sqrt{1-2x-3x^2} = \sqrt{1-3x} \cdot \sqrt{1+x}  = \left( \sum_{i\ge0} \binom{1/2}{i} (-3x)^i \right) \left( \sum_{j\ge0} \binom{1/2}{j} x^j \right) .\]
Thus for $n\ge0$,
\[ C_{3,n}^0 = -\frac12 \sum_{0\le i\le n+2} (-3)^i\binom{1/2}{i} \binom{1/2}{n+2-i}. \]
The right hand side of this equation motivates the definition
\begin{equation}\label{eq:C3n}
\defcolor{C_{3,n}^0} :=
\begin{cases} \frac12, & n=-1, \\
-\frac12, & n=-2, \\
0 , & n\le-3, 
\end{cases}
\end{equation}
On the other hand, we have
\[ \sqrt{1-2x-3x^2} = \sum_{i\ge0} \binom{1/2}{i} (-x)^i(2+3x)^i = \sum_{i\ge0} \binom{1/2}{i} (-x)^i \sum_{0\le j\le i} \binom{i}{j} 2^{i-j}3^jx^j. \]
Thus 
\[ C_{3,n}^0 = -\frac12 \sum_{0\le j\le (n+2)/2} (-1)^{n+2-j} 2^{n+2-2j} 3^j \binom{1/2}{n+2-j} \binom{n+2-j}{j}.\]

Applying Proposition~\ref{prop:Ckd} to the generating function $C_3^0(x)$ gives 
\begin{equation}\label{eq:C3}
 C_3^d(x) = \frac{x}{1-\frac{x}{1- \cdots \frac{x}{1-C_3^{0}(x) }} } 
\end{equation}
where the number of ones is $d$.
Equation~\eqref{eq:C3} appeared in work of Mansour and Vainshtein~\cite{MansourVainshtein} and it is a special case of the generating function studied by Flajolet~\cite{Flajolet} for \demph{labeled positive paths}.
Such a path $L$ starts at $(0,0)$ and stays weakly above the line $y=0$, with three kinds of steps $U=(1,1)$, $D=(1,-1)$, and $H=(1,0)$.
Each step is labeled with some weight, and the \demph{total weight} of $L$ is the sum of all weights of the steps.
The \demph{height} of $L$ is the largest $y$-coordinate of a point on $L$.

By Equation~\eqref{eq:C3} or Remark~\ref{rem:LimitCkd}, the number $C_{3,n}^d$ interpolates between the Motzkin number $C_{3,n}^0$ and the Catalan number $C_n=\lim_{d\to\infty} C_{3,n}^d$.
For $d=1,\ldots,5$, the sequences $\{C_{3,n}^d\}$ are recorded in The OEIS~\cite[A005773, A054391--A054394]{OEIS}. 
For an arbitrary $d$, Barcucci, Del Lungo, Pergola, and Pinzani~\cite{CatPerm} studied $C_{3,n}^d$ in terms of permutations avoiding certain \demph{barred patterns}.

\begin{proposition}
For $d,n\ge0$ the number $C_{3,n}^d$ enumerates the following families of objects.
\begin{itemize}
\item Labeled positive paths with total weight $n$ and no $H$-step strictly below $y=d$, where each $U$-step or $D$-step weakly below $y=d$ has a weight $1/2$ and each other step has a weight $1$.
\item Permutations of $1,2,\ldots,n$ avoiding $321$ and $(d+3)\bar1(d+4)23\cdots(d+2)$ (barred pattern).
\end{itemize}
\end{proposition}

\begin{proof}
A specialization of work of Flajolet~\cite[Thm. 1]{Flajolet} gives the first family of objects enumerated by $C_{3,n}^d$.
The second one follows from Barcucci, Del Lungo, Pergola, and Pinzani~\cite[(12)]{CatPerm}.
\end{proof}

Barcucci, Del Lungo, Pergola, and Pinzani~\cite[p. 47]{CatPerm} provided a closed formula for $C_3^d(x)$ but no formula for $C_{3,n}^d$. 
We will provide a different closed formula for $C_3^d(x)$ and derive a closed formula for $C_{3,n}^d$ from that.

\begin{theorem}\label{thm:C3d}
For $d\ge0$ we have
\[ C_3^d(x) = \frac{2xF_{d+1}(x)F_{d+2}(x)-x^d-x^{d+1} + x^d\sqrt{1-2x-3x^2}}{2(F_{d+2}(x)^2-x^d-x^{d+1})}. \]
\end{theorem}

\begin{proof}
We induct on $d$. 
The result is trivial if $d=0$. 
For $d\ge1$, it follows from Proposition~\ref{prop:Ckd} and the induction hypothesis that
\begin{align*}
C_3^d(x) &= \frac{x}{1-C_3^{d-1}(x)} \\
&= \frac{2x(F_{d+1}(x)^2-x^{d-1}-x^{d}) }{2F_{d+1}(x)^2-2x^{d-1}-2x^d-2xF_{d}(x)F_{d+1}(x)+x^{d-1}+x^{d} - x^{d-1}\sqrt{1-2x-3x^2}} \\
&= \frac{2x(F_{d+1}(x)^2-x^{d-1}-x^{d}) } {2F_{d+1}(x)F_{d+2}(x)-x^{d-1}-x^d-x^{d-1}\sqrt{1-2x-3x^2}} \\
&= \frac{2x(F_{d+1}(x)^2-x^{d-1}-x^{d}) (2F_{d+1}(x)F_{d+2}(x)-x^{d-1}-x^d + x^{d-1}\sqrt{1-2x-3x^2}) } {(2F_{d+1}(x)F_{d+2}(x)-x^{d-1}-x^d)^2-x^{2d-2}(1-2x-3x^2)} \\
&= \frac{(2F_{d+1}(x)^2-2x^{d-1}-2x^{d})(2xF_{d+1}(x)F_{d+2}(x)-x^{d}-x^{d+1} + x^{d}\sqrt{1-2x-3x^2})}{4F_{d+1}(x)^2 F_{d+2}(x)^2-4(x^{d-1}+x^d)F_{d+1}(x)F_{d+2}(x) +4x^{2d-1}+4x^{2d} }.
\end{align*}
This implies the desired expression of $C_3^d(x)$ since
\begin{align*}
& (F_{d+1}(x)^2-x^{d-1}-x^{d}) (F_{d+2}(x)^2-x^d-x^{d+1}) \\
=& F_{d+1}(x)^2 F_{d+2}(x)^2 - (x^{d-1}+x^d) (F_{d+2}(x)^2+xF_{d+1}^2(x)) + x^{2d-1}+2x^{2d}+x^{2d+1} \\
=& F_{d+1}(x)^2 F_{d+2}(x)^2 - (x^{d-1}+x^d) F_{d+2}(x)F_{d+1}(x) +x^{2d-1}+x^{2d}.
\end{align*}
Here the last step follows from
\[ F_{d+2}(x)^2+xF_{d+1}^2(x) = F_{d+2}(x)F_{d+1}(x)-xF_{d+2}(x)F_d(x) + xF_{d+1}^2(x)
= F_{d+2}(x)F_{d+1}(x)+x^{d+1} \]
as one can show $F_{d+1}(x)^2-F_d(x) F_{d+2}(x) = x^d$ by induction on $d$.
\end{proof}

Some examples are given below.
\[ C_3^1(x) 
= \frac{x-3x^2+x\sqrt{1-2x-3x^2}}{2(1-3x)} \]
\[ C_3^2(x) = \frac{2x-7x^2+3x^3 +x^2\sqrt{1-2x-3x^2}}{2(1-4x+3x^2-x^3)} \] 
\[ C_3^3(x) = \frac{2x-10^2+13x^3-5x^4 +x^3\sqrt{1-2x-3x^2}}{2(1-6x+11x^2-7x^3)} \] 

Next, we derive a closed formula for $C_{3,n}^d$ from the expression of $C_3^d(x)$ given by Theorem~\ref{thm:C3d}.

\begin{theorem}\label{thm:C3dn}
For $d\ge1$ and $n\ge0$ we have
\begin{align*}
C_{3,n}^d = \sum_{ \substack{ \alpha\models n+1 \\ h>1\Rightarrow \alpha_h\le d+1} } 
& -\left( C_{3,\alpha_1-d-2}^0 
+ \frac{\delta_{\alpha_1,d}}2 +  (-1)^{\alpha_1}\sum_{i+j=\alpha_1-1} \binom{d-i}{i} \binom{d+1-j}{j} \right) \\
& \cdot \prod_{h\ge2} \left( \left(\delta_{\alpha_h,d}+(-1)^{\alpha_h-1}\sum_{i+j=\alpha_h} \binom{d+1-i}{i} \binom{d+1-j}{j} \right) \right) 
\end{align*}
where $C_{3,m}^0$ is the Motzkin number when $m\ge0$ or defined by Equation~\eqref{eq:C3n} when $m<0$ and 
\[ \defcolor{\delta_{m,d}}:=
\begin{cases} 
1, & m\in\{d,d+1\}, \\ 
0, & \text{otherwise.} 
\end{cases} \]
\end{theorem}

\begin{proof}
We have
\[ F_r(x)F_s(x) = \sum_{0\le i\le (r-1)/2} (-x)^i\binom{r-1-i}{i} \sum_{0\le j\le (s-1)/2} (-x)^j\binom{s-1-j}{j}, \]
\[ \sqrt{1-2x-3x^2} = \sqrt{1-3x} \cdot \sqrt{1+x}  = \left( \sum_{i\ge0} \binom{1/2}{i} (-3x)^i \right) \left( \sum_{j\ge0} \binom{1/2}{j} x^j \right) .\]
Hence
\begin{align*}
& 2xF_{d+1}(x)F_{d+2}(x)-x^d-x^{d+1} + x^d\sqrt{1-2x-3x^2} \\
=& \sum_{n\ge1} x^{n} \left( -\delta_{n,d} +  (-1)^{n-1}2\sum_{i+j=n-1} \binom{d-i}{i} \binom{d+1-j}{j} + \sum_{i+j=n-d} (-3)^i \binom{1/2}{i} \binom{1/2}{j} \right) ,
\end{align*}
\[ F_{d+2}(x)^2-x^d-x^{d+1} = 1 - \sum_{1\le n\le d+1} x^{n} \left(\delta_{n,d}+(-1)^{n-1}\sum_{i+j=n} \binom{d+1-i}{i} \binom{d+1-j}{j} \right) .\] 
Substituting these expressions in the formula for $C_3^d(x)$ given by Theorem~\ref{thm:C3d} and extracting the coefficient of $x^{n+1}$ we obtain the desired formula for $C_{3,n}^d$. 
\end{proof}

We have not found the sequences $\{ C_{k,n}^d \}$ for $k\ge4$ and $d\ge2$ in the literature.


\subsection{The case $d=2$ and $e=\ell=1$}\label{sec:d=2}
As Proposition~\ref{prop:Ckd} gives a way to obtain the numbers $C_{k,n}^{d+1}$ from $C_k^d(x)^m$, we investigate the sequences $[x^{n+m}]C_k^d(x)^m$ for fixed $d$.

We begin with $d=0$ and generalize the closed formulas~\cite[(9) and (11)]{CatMod} for $C_k^{0}(x) = M_{k-1}(x)$.
To state our result, we review some notation below.

Let $\lambda=(\lambda_1,\ldots,\lambda_n)$ be a partition with $m_i$ parts equal to $i$ for $i=0,1,2,\ldots$.
Then 
\begin{itemize}
\item
$|\lambda|=n$ if and only if $m_1+2m_2+\cdots+km_k=n$, and 
\item
$\lambda\subseteq k^n$ if and only if $m_0+\cdots+m_k=n$ and $m_{k+1}=m_{k+2}=\cdots=0$.
\end{itemize}
The \demph{monomial symmetric function} $m_\lambda(x_1,\ldots,x_n)$ is the sum of $x_1^{a_1}\cdots x_n^{a_n}$ for all rearrangement $(a_1,\ldots,a_n)$ of $(\lambda_1,\ldots,\lambda_n)$. 
Taking $x_1=\cdots=x_n=1$ in $m_\lambda$ gives the multinomial coefficient 
\[ m_\lambda(1^n) = {n\choose m_0,m_1,m_2,\ldots}. \]
One sees that
\begin{equation}\label{eq:monomial}
\prod_{1\le i\le n} (1+x_i+x_i^2+\cdots+x_i^k)^n = \sum_{\lambda\subseteq k^n} m_\lambda(x_1,\ldots,x_n). 
\end{equation}

\begin{proposition}\label{prop:l=0}
For $k,m\ge1$ and $n\ge0$, the number of plane forests with $m$ components and $n+m$ total nodes, each of degree less than $k$, is 
\begin{eqnarray*}
[x^{n+m}] M_{k-1}(x)^m 
&=& \frac{m}{n+m} \sum_{0\le j\le n/k} (-1)^j {n+m\choose j} {2n+m-jk-1\choose n+m-1} \\
&=& \frac{m}{n+m} \sum_{ \substack{ \lambda\subseteq(k-1)^{n+m} \\ |\lambda|=n } } m_\lambda(1^{n+m}). 
\end{eqnarray*}
\end{proposition}

\begin{proof}
The result follows from~\cite[Proposition~4.5 and (12)]{CatMod}. 
For completeness we include a direct proof here.
Plane forests with $m$ components and $n+m$ total nodes, each of degree less than $k$, are enumerated by $[x^{n+m}]M_{k-1}(x)^m$.
We have $M_{k-1}(x) = x( 1 - M_{k-1}(x)^k ) / (1-M_{k-1}(x) )$ by~\cite[(5)]{CatMod}.
Applying Lagrange inversion to $A(x) = x(1-x)/(1-x^k)$ and $B(x) = M_{k-1}(x)$ gives
\begin{eqnarray}
 [x^n](M_{k-1}(x))^m &=& \frac{m}{n} [x^{n-m}] \frac{(1-x^k)^{n}} {(1-x)^n} \label{eq:l=0a} \\
 &=& \frac{m}{n} [x^{n-m}] \left(1+x+x^2+\cdots+x^{k-1}\right)^n. \label{eq:l=0b}
\end{eqnarray}
Applying binomial expansion to \eqref{eq:l=0a} and replacing $n$ with $n+m$ gives the first formula.
Applying \eqref{eq:monomial} to \eqref{eq:l=0b} and replacing $n$ with $n+m$ gives the second formula.
\end{proof}

\begin{remark}
For $k=1$ we have $M_{k-1}(x) = x$. Thus $\{[x^{n+m}] M_{k-1}(x)^m\} = \{1,0,0,\ldots\}$ for any $m\ge0$.
For $k=2$ we have $M_{k-1}(x) = x/(1-x)$. Thus $[x^{n+m}] M_{k-1}(x)^m = {m+n-1\choose n}$ for $m,n\ge0$.
For $k=3$ the sequences $\{[x^{n+m}](M_{k-1}(x)^m\}$ form the diagonals of the Motzkin triangle~\cite[A026300]{OEIS}; see also~\cite[A002026, A005322--A005325]{OEIS} for $m=2,\ldots,6$.
We have not found any result in OEIS for $k\ge4$ and $m\ge2$.
\end{remark}

We next generalize the closed formulas~\cite[(9) and (11)]{CatMod} for $C_{k,n}^{1} = C_{k,n}$.

\begin{proposition}\label{prop:d=1}
For $k,m,n\ge1$, the number of plane forests with $m$ components and $n$ non-root nodes, each of degree less than $k$, is 
\begin{eqnarray*}
[x^{n+m}] C_k(x)^m &=& \frac{m}{n} \sum_{0\le j\le (n-1)/k} (-1)^j {n\choose j} {2n+m-jk-1\choose n+m} \\
&=& \sum_{\lambda\subseteq(k-1)^n} \frac{n-|\lambda|}{n}{m+n-|\lambda|-1\choose n-|\lambda|} m_\lambda(1^n).
\end{eqnarray*}
\end{proposition}

\begin{proof}
Combining Proposition~\ref{prop:Ckd} with \eqref{eq:l=0a} we have
\begin{eqnarray}
[x^{n+m}] C_k(x)^m 
&=& \sum_{1\le i\le n} { m+i-1 \choose i} \frac{i}{n} [x^{n-i}] \frac{(1-x^k)^n} {(1-x)^{n} } \label{eq:d=1a} \\
&=& \sum_{1\le i\le n} \frac{m}{n} { m+i-1\choose i-1} [x^{n-i}] \frac{(1-x^k)^n} {(1-x)^{n} } \nonumber \\ 
&=& \frac{m}{n} [x^{n-1}] \frac{(1-x^k)^n} { (1-x)^{n+m+1} }. \label{eq:l=1c}
\end{eqnarray} 
Applying binomial expansion to \eqref{eq:l=1c} gives the first formula.
Applying \eqref{eq:monomial} to \eqref{eq:d=1a} gives the second formula.
\end{proof}

A \demph{weak composition} of $n$ into $m$ parts is a sequence of $m$ nonnegative integers whose sum is $n$.
Our next result shows that the sequences $\{[x^{n+m}]C_k(x)^m\}_{n\ge0}$ are related to weak compositions for $k=1,2$.
The case $k=1$ is well known and the case $k=2$ has been studied by Janji\'c and Petkovi\'c~\cite{WeakComp} with a different approach from ours.
For $k=2$ and $1\le m\le 10$ see also~\cite[A011782, A045623, A058396, A062109, A169792--A169797]{OEIS}.
We have not found any result in the literature for $k\ge3$ and $m\ge2$ except the case $k=3$ and $m=2$~\cite[A036908]{OEIS}.

\begin{corollary}\label{cor:l=1}
For $m,n\ge0$, weak compositions of $n$ into $m$ parts are enumerated by 
\[ [x^{n+m}] C_1(x)^m = {m+n-1\choose n} \]
and weak compositions of $n$ with $m-1$ zero parts are enumerated by
\[ [x^{n+m}] C_2(x)^m = \sum_{0\le i\le m} {m\choose i} {n-1\choose n-i} 2^{n-i}
= \sum_{0\le i \le n} {m+i-1 \choose i} {n-1\choose n-i}.\]
\end{corollary}

\begin{proof}
By Proposition~\ref{prop:d=1}, weak compositions of $n$ into $m$ parts are enumerated by 
\[ [x^{n+m}]C_1(x)^m = {m+n-1\choose n} \]
since they are in bijection with plane forests with $m$ components and $n$ non-root nodes, each of degree less than $k=1$.
In fact, such a forest is completely determined by the numbers of non-root nodes in its components.   
This implies the above bijection.

Similarly, weak compositions of $n$ with $m-1$ zero parts are enumerated by $[x^{n+m}]C_2(x)^m$, since they are in bijection with plane forests with $m$ components and $n$ non-root nodes, each of degree less than $k=2$. 
To see this bijection, let $v_1,\ldots,v_r$ be the children of the roots of such a forest.
Since non-root nodes have degree at most one, the maximal subtree rooted at each $v_i$ is a path consisting of $a_i$ nodes, and this forest is determined by $a_1,\ldots,a_r$.
We have $a_1+\cdots+a_r=n$ and thus $(a_1,0,a_2,0,\ldots,a_{r-1},0,a_r)$ is a weak composition of $n$ with $m-1$ zero parts. 

Now using Proposition~\ref{prop:Ckd} and Proposition~\ref{prop:k=2} we have
\begin{eqnarray*}
[x^{n+m}]C_2(x)^m &=& [x^n] \frac{(1-x)^m}{(1-2x)^m} \\
&=& [x^n] \left(1+\frac{x}{1-2x}\right)^m \\
&=& \sum_{0\le i\le m} {m\choose i} [x^{n-i}](1-2x)^{-i} \\
&=& \sum_{0\le i\le m} {m\choose i} {n-1\choose n-i} 2^{n-i} .
\end{eqnarray*}
The second formula of $[x^{n+m}]C_2(x)^m$ follows from \eqref{eq:d=1a}.
\end{proof}


Now we study the case $d=2$.

\begin{proposition}
For $m,n\ge0$ and $k\ge1$ we have
\begin{eqnarray*}
[x^{n+m}] C_k^{2}(x)^m &=& {m+n-1\choose n} + \sum_{1\le i\le n-1} {m+i-1\choose i} \frac{i}{n-i} \sum_{0\le j\le \frac{n-i-1}k} (-1)^j {n-i\choose j} {2n-i-jk-1\choose n} \\
&=& {m+n-1\choose n} + \sum_{1\le i\le n-1} {m+i-1\choose i} \sum_{\lambda\subseteq(k-1)^{n-i}} \frac{n-i-|\lambda|}{n-i}{n-|\lambda|-1\choose n-|\lambda|-i} m_\lambda(1^{n-i}).
\end{eqnarray*}
In particular,
\begin{eqnarray*}
C_{k,n}^{2} &=& 1+ \sum_{1\le i\le n-1} \frac{i}{n-i} \sum_{0\le j\le (n-i-1)/k} (-1)^j {n-i\choose j} {2n-i-jk-1\choose n} \\
&=& 1 + \sum_{1\le i\le n-1} \sum_{\lambda\subseteq(k-1)^{n-i}} \frac{n-i-|\lambda|}{n-i}{n-|\lambda|-1\choose n-|\lambda|-i} m_\lambda(1^{n-i}).
\end{eqnarray*}
\end{proposition}

\begin{proof}
Proposition~\ref{prop:Ckd} implies
\begin{eqnarray*}
[x^{n+m}] C_k^{2}(x)^m &=& \sum_{0\le i\le n} {m+i-1\choose i} [x^n] C_k^{1}(x)^i \\
&=& {m+n-1\choose n} + \sum_{1\le i\le n-1} {m+i-1\choose i} [x^n] C_k^{1}(x)^i. 
\end{eqnarray*}
Substituting formulas from Proposition~\ref{prop:d=1} establishes the result.
\end{proof}


We next study the special case when $d=k=2$.

\begin{proposition}
Let $m,n\ge0$. Then
\[ [x^{n+m}] C_2^{2}(x)^m = \sum_{0\le i\le n} {m+i-1\choose i} \sum_{0\le j\le n-i}{i+j-1\choose j}{n-i-1\choose n-i-j}.\]
In particular, 
\[ C_{2,n}^{2} = \sum_{0\le j\le n} {n+j-1\choose 2j} = F_{2n-1}.\]
\end{proposition}
\begin{proof}
Let $m,n\ge0$.
Proposition~\ref{prop:k=2} gives $C_1^{2}(x)=C_2^{1}(x)$ and thus Corollary~\ref{cor:l=1} implies 
\[ [x^{n+m}] C_1^{2}(x)^m 
= \sum_{0\le i \le n} {m+i-1 \choose i} {n-1\choose n-i}.\]
Combining this with Proposition~\ref{prop:Ckd} we have
\[ [x^{n+m}] C_2^{2}(x)^m = \sum_{0\le i\le n} {m+i-1\choose i} \sum_{0\le j\le n-i}{i+j-1\choose j}{n-i-1\choose n-i-j}. \]
In particular, taking $m=1$ we have
\begin{eqnarray*}
C_{2,n}^{2} &=& \sum_{0\le i\le n} \sum_{0\le j\le n-i}{i+j-1\choose j}{n-i-1\choose n-i-j} \\
&=& \sum_{0\le j\le n} \sum_{0\le i\le n-j} {i+j-1\choose j} {n-i-1\choose n-i-j} \\
&=& \sum_{0\le j\le n} {n+j-1\choose 2j}
\end{eqnarray*}
where the last step follows from choosing $2j$ elements from $[n+j-1]$, assuming the $(j+1)$th chosen element is $i+j$ for some $i$.
This sum is known to be the Fibonacci number $F_{2n-1}$, or one can use Proposition~\ref{prop:k=2} and the discussion in Section~\ref{sec:nil} to conclude that $C^2_{2,n} = C^3_{1,n} = F_{2n-1}$.
\end{proof}

\begin{remark}
The limit of the sequence $\{[x^{n+m}]C_k^d(x)^m\}$ as $k\to\infty$ or $d\to\infty$ is $\{ [x^{n+m}] C(x)^m \}$ where
$C(x) := \sum_{n\ge0} C_n x^{n+1}$.
It is well known that $C(x) = x/(1-C(x))$.
Thus for $m\ge1$ and $n\ge0$ applying Lagrange inversion to $A(x) := x(1-x)$ and $B(x) :=C(x) $ gives
\[ [x^{n+m}] C(x)^m = \frac{m}{n+m}[x^{n}] (1-x)^{-(n+m)} = \frac{m}{n+m}{2n+m-1\choose n}. \] 
Hence the sequences $\{ [x^{n+m}] C(x)^m \}$ form the diagonals of Catalan's triangle~\cite[A009766]{OEIS}.
\end{remark}

\section{Questions and Remarks}

The operad theory provides another perspective on the study of the nonassociativity of a binary operation; see, for example, recent work~\cite{MagmaticOperad, CharOp}.
In fact, there is a nonsymmetric operad with respect to each binary operation $*$ and its Hilbert series is $\sum_{n\ge0} C_{*,n} x^{n+1}$.
It would be interesting to study the operads associated with the binary operations considered in this paper.

One can also replace the single operation $*$ in the expression $x_0*x_1*\cdots*x_n$ with multiple different operations defined on the underlying set $X$ and enumerate nonequivalent results obtained by inserting parentheses.

Finally, we remark that it is common in combinatorics to study generalizations of important integer sequences. 
For instance, instead of the Catalan numbers, one may consider the \demph{Stirling number of the second kind} $S(n,k)$, which counts partitions of the set $[n]:=\{1,2,\cdots,n\}$ into $k$ (unordered) blocks.
Another way to understand $S(n,k)$ is through commutation relations and normal ordering, as addressed in the recent monograph by Mansour and Schork~\cite{MansourSchork}.
If $U$ and $V$ are two operators satisfying the relation $UV-VU=1$ then $(VU)^n = \sum_{k=0}^n S(n,k) V^k U^k$~\cite[Theorem~3.5]{MansourSchork} and extensive work has been done on the expansion of more general expressions such as $(V^r U^s)^n$, giving rise to various generalizations of the Stirling numbers~\cite[\S4.1]{MansourSchork}.
Our variations of the Catalan numbers follow a somewhat similar (although directly related) idea.

\end{document}